\documentclass[11pt]{article}
\usepackage{a4wide}

\usepackage[a4paper, margin=2.3cm]{geometry}
\usepackage{amsmath,amssymb,amsthm, comment, enumitem, array, colonequals}
\usepackage{color}
\usepackage{parskip}
\usepackage{hyperref}
\usepackage{parskip}
\usepackage{setspace}
\usepackage{graphicx}
\usepackage{mathtools}
\usepackage[thinc]{esdiff}
\usepackage[euler]{textgreek}

\usepackage{babel}

\usepackage[capitalise, nameinlink]{cleveref}
\crefname{equation}{}{}
\crefname{figure}{{\sc Figure}}{{\sc Figure}}
\crefname{subsection}{Subsection}{Subsections}

\newtheorem{theorem}{Theorem}[section]
\newtheorem{proposition}[theorem]{Proposition}
\newtheorem{remark}[theorem]{Remark}
\newtheorem{corollary}[theorem]{Corollary}
\newtheorem{lemma}[theorem]{Lemma}
\newtheorem{question}[theorem]{Question}

\newtheorem{conjecture}[theorem]{Conjecture}
\newtheorem*{definition*}{Definition}
\newtheorem{exam}{Example}

\hypersetup{
    colorlinks=true,
    linkcolor = blue,
    citecolor=magenta,
    urlcolor=cyan,
}

\def\F{\mathcal{F}}
\def\F{\mathcal{F}}

\def\Fbb{\mathbb{F}}

\def \F{{\mathbb F}}

\def \1{\textbf{1}}

\DeclareMathOperator{\SL}{SL}

\def\l({\left(}
\def\r){\right)}

\def\lv{\left\vert}
\def\rv{\right\vert}

\def\lo{\left\{}
\def\ro{\right\}}

\def\rar{\rightarrow}

\begin{document}
\title{Sets preserved by a large subgroup of the special linear group}

\author{Le Quang Hung\thanks{Hanoi University of Science and Technology, Hanoi. Email: hung.lequang@hust.edu.vn}\and Thang Pham\thanks{University of Science, Vietnam National University, Hanoi. Email: thangpham.math@vnu.edu.vn}\and Kaloyan Slavov\thanks{Department of Mathematics, ETH Z\"{u}rich. Email: kaloyan.slavov@math.ethz.ch}}
\maketitle

\begin{abstract}
Let $E$ be a subset of the affine plane over a finite field $\F_q$. We bound the size of the subgroup of $\SL_2(\F_q)$ that preserves $E$. As a consequence, we show that if $E$ has size $\ll q^\alpha$ and is preserved by $\gg q^\beta$ elements of $\SL_2(\F_q)$ with $\beta\geq 3\alpha/2$, then $E$ is contained in a line. This result is sharp in general, and will be proved by using combinatorial arguments and applying a point-line incidence bound in $\mathbb{F}_q^3$ due to Mockenhaupt and Tao (2004).  
\end{abstract}

\section{Introduction}
Let $\F_q$ be a finite field of order $q$, let $E$ be a subset of $\F_q^2$, and let $R$ be a subset of $\SL_2(\F_q)$. In a recent work \cite{HHIP} on packing sets over finite fields, Hegyv\'{a}ri, Hung, Iosevich, and Pham give lower bounds for the size of
\[ R (E) := \bigcup_{\theta \in R} \theta (E) = \lo \theta (x) \colon x\in E, \theta \in R \ro \]
in terms of $|E|,|R|, q$ (and other quantities determined by $R$ and $E$). However, in some cases, only the trivial bound $|R(E)|\geq |E|$ is available (see Proposition 2.1, Theorem 2.2, and Theorem 2.3 in \cite{HHIP}). 

In the present paper, we ask for a description of extremal configurations, that is, sets $E\subset \F_q^2$ which ``stay packed" under $\SL_2(\F_q)$. More precisely, our goal is to study sets $E\subset\F_q^2$ such that $R(E)=E$ for a large subset $R$ of $\SL_2(\F_q)$. 

Throughout the paper, $X\ll Y$ means $X\leq CY$ for a constant $C$; $X\gg Y$ means $X\geq C Y$ for a constant $C$, and $X\approx Y$ means $X\ll Y$ and $Y\ll X$.

For a subset $E$ of $\F_q^2$, let 
\[R_E=\{\theta\in\SL_2(\F_q)\ |\ \theta(E)=E\}.\]

\begin{exam}
When $E$ is empty, $\{(0,0)\}$, all of $\F_q^2$, or $\F_q^2-\{(0,0)\}$, we have that $R_E=\SL_2(\F_q)$ has size $\approx q^3$.  
\end{exam}

\begin{exam}\label{Exa:basic_line}
If $E$ is a line, then $|E|=q$ and $|R_E|=q^2-q\approx q^2$.
\end{exam}

Note that for any $E\subset\F_q^2$, denoting the complement $\F_q^2-E$ by $E^c$, we have $R_E=R_{E^c}$. So Example \ref{Exa:basic_line} also shows that when $E$ is the complement of a line, we have $|R_E|\approx q^2.$

\begin{exam}
   For positive divisor $c$ of $q-1$, there exists a subgroup $S \subset \Fbb_q^\ast$ with $|S| = (q-1)/c$. Let 
\[ E = \lo  (0,y) \in \Fbb_q^2 \colon y \in S \ro. \]
Then 
\[ R_E = \lo  \begin{pmatrix}
    a & 0 \\
    d & a^{-1}
\end{pmatrix} \in \SL_2(\Fbb_q) \colon a \in S , d \in \Fbb_q \ro .\]
Here again $|E|= |S| = (q-1)/c\approx q$ and $|R|= q|S| = q(q-1)/c\approx q^2$. 
\end{exam}


Motivated by these examples, we ask the following

\begin{question}
Let $E$ be a subset of $\F_q^2$ that is preserved under ``many" elements in $\SL_2(\F_q)$. Does it follow (under an assumption on $|E|$) that $E$ is contained in a line?
\end{question}

\begin{remark}
There is a significant difference between the cases when the ground field is a prime field versus a general finite field. In this paper, we treat these cases separately. More precisely, while tools from group theory would be enough to study the prime field setting, for the arbitrary field setting we use a combination of combinatorial arguments and a point-line incidence bound due to Mockenhaupt and Tao \cite{MT04}.
\end{remark}

We answer in the affirmative. 

\begin{theorem}\label{main-thm}
Given $c_1>0$, there exists a constant $c_2>0$ with the following property. Let $\alpha,\beta\geq 0$ be such that $\beta\geq 3\alpha/2$. Let $E$ be a subset of $\F_q^2$ with $|E|\leq c_1 q^{\alpha}$. Suppose the set $R_E$ of matrices $\theta$ in $\SL_2(\Fbb_q)$ such that $\theta (E) = E$ satisfies $|R_E|\geq c_2 q^\beta$. Then $E$ is contained in a line. 
\end{theorem}

Note that the assumption involves merely counting data for $|E|$ and $|R_E|$ while the conclusion is a geometric structure on $E$. 
In the next corollary, we spell out some illustrative special cases.

\begin{corollary}\label{main-thm'} Let $E$ be a subset of $\F_q^2$. Let $R_E$ denote the set of matrices $\theta$ in $\SL_2(\Fbb_q)$ such that $\theta (E) = E$. 
\begin{itemize}
\item[(i)]If $|E|\approx q^{1/2}$ and $|R_E|\gg q^{3/4}$, then $E$ is contained in a line;
    \item[(ii)] If $|E|\approx q^{2/3}$ and $|R_E|\gg q$, then $E$ is contained in a line;
    \item[(iii)]  If $|E|\approx q$ and $|R_E|\gg q^{3/2}$, then $E$ is contained in a line;
    \item [(iv)] If $|E|\approx q^{4/3}$ and $|R_E|\geq q^{2}$, then $E$ is contained in a line (in particular $|E|\leq q$).
   \end{itemize}
\end{corollary}

The key ingredient in the proof of Theorem \ref{main-thm} is the next result, which is of independent interest. 

\begin{theorem}
Let $E\subset \Fbb_q^2$. Let $R_E$ be the set of matrices $\theta$ in $\SL_2(\Fbb_q)$ such that $\theta (E) = E$. Suppose at least $2$ lines through the origin intersect $E-\{(0,0)\}$.
Then $|R_E|\ll |E|^{3/2}$.

    \label{Prop:main_bound_set_of_points}
\end{theorem}

The main difficulty in proving this theorem arises when the number of lines depends on $q$. 

The exponent $3/2$ in Theorem \ref{Prop:main_bound_set_of_points}
is sharp, as the example below illustrates.

\begin{exam}\label{Exa:F_p}
    For $p$ a prime and $q=p^2$, let $E= \Fbb_p^2\subset\F_q^2$. Then $R_E\supset \SL_2(\Fbb_p)$. So $|E|=q$ and $|R_E|\geq p^3 -p \approx q^{3/2} =|E|^{3/2}$.
\end{exam}

In fact, Example \ref{Exa:F_p} shows that Theorem \ref{main-thm} is sharp when $\alpha=1$ in the following sense: in the statement of Corollary \ref{main-thm'}(i), if one keeps the assumption $|E|\ll q$ but replaces $|R_E|\gg q^{3/2}$ by $|R_E|\gg q^{3/2-\varepsilon}$ with any $\varepsilon>0$, the statement would be false. 

The example below shows that Theorem \ref{main-thm} is in fact sharp in the sense described above for all $\alpha$ of the form $\alpha=1/s$ or $\alpha=2/s$ for a positive integer $s$.  

\begin{exam}\label{Exa_alphas}
Let $\alpha$ be of the form $\alpha=1/s$ or $\alpha=2/s$ for a positive integer $s$. Then there exist sufficiently large $q$ and subsets $E\subset\F_q^2$ of size $|E|\approx q^\alpha$ such that $|R_E|\approx q^{3\alpha/2}$ 
and $E$ is not contained in a line. Indeed, pick any positive integer $r'$, set $r=(2/\alpha)r'$, and consider $E=\F^2_{p^{r'}}\subset\F_q^2$, where $q=p^r$. Then $|E|=p^{2r'}=q^\alpha$ while $R_E$ contains $\SL_2(\F_{p^{r'}})$,
so $|R_E|\geq p^{3r'}-p^{r'}\approx q^{3\alpha/2}$.    
\end{exam}

The result below addresses in particular the case of a prime ground field. 

\begin{proposition}
Let $q=p^r$, where $r\geq 1$ and $p$ is prime.
Suppose $E$ is not $\F_q^2-\{(0,0)\}$ or all of $\F_q^2$.
Suppose at least $2$ lines through the origin intersect $E-\{(0,0)\}$. Then $|R_E|\leq p^{r-1}|E|$. 
\label{Prop:gp_theory}
\end{proposition}

When $|E|\approx q^{\alpha}$, Proposition \ref{Prop:gp_theory} gives the bound  $|R_E|\leq q^{\alpha+(r-1)/r}$. Compared to Theorem \ref{Prop:main_bound_set_of_points}, this result is only interesting when $\alpha\geq 2-2/r$, which will be a very small range when $r$ is large.

\begin{remark}
Given $E\subset\F_q^2$, one can obtain an upper bound for $R_E$ as follows. 
Let $P=E\times E \subseteq \mathbb{F}_q^2\times \mathbb{F}_q^2$. Then, we know from \cite{HHIP} that
    \[\lv I(P,R_E)- \frac{\lv P\rv \lv R_E\rv}{q^2} \rv \ll q\sqrt{\lv R_E\rv \lv P\rv}+|R_E|.\]
If $1\ll |E|$ and $1\ll |E^c|$, using the fact that $I(P, R_E)=|R_E||E|$, one has $|R_E|\ll  q^2$ (suppose first $|E|<q^2/2$; then apply the same for the complement $E^c$).
However, the bound from Theorem \ref{Prop:main_bound_set_of_points} is better as long as $|E|\ll q^{4/3}$. 
\label{Rem:q^2_general_bound}
\end{remark}

It would be interesting to study analogues of Theorem \ref{main-thm} in higher dimensions. 

Based on Example \ref{Exa_alphas}, we conjecture that 
Theorem \ref{main-thm} is sharp for each specific $\alpha\in [0,1]$.

In light of results on the packing set problems over the reals (see, for example, \cite{IM} and references therein) Theorem \ref{main-thm} suggests that it is reasonable to state the following conjecture. 

\begin{conjecture}
     Let $E$ be a compact set in $\mathbb{R}^2$ with $\dim_H(E)=\alpha$. Assume that there exists a set $R$ of matrices $\theta$ in $\SL_2(\mathbb{R})$ such that
     $\theta(E)=E$ for each $\theta\in R$, with
     $\dim_H(R)> 3\alpha/2$. Then the set $E$ is contained in a line. Here $\dim_H$ means the Hausdorff dimension. 
\end{conjecture}

\section{Proof of Theorem \ref{main-thm}}

Let $P$ be a set of points and $L$ be a set of lines in $\mathbb{F}_q^3$. We denote the number of incidences between $P$ and $L$ by $I(P, L)$. The proof of Theorem \ref{main-thm} makes use of an incidence bound $I(P, L)$. A modification of the proof of Theorem 8.6 in  \cite{MT04} gives the following result.

\begin{theorem}[Mockenhaupt--Tao]
\label{lem4.2} 
    Let $P$ be a set of points in $\Fbb_p^3$ and $L$ be a set of lines in $\Fbb_p^3$. Assume that any plane contains at most $M$ lines from $L$, then we have 
    \[I(P, L)\le c\left( |P|^{1/2}|L|^{3/4}M^{1/4}+|P|+|L|\right),\]
with an absolute constant $c$.
\end{theorem}

The lemma below will allow us to apply the point--line incidence bound of Theorem \ref{lem4.2}. 

\begin{lemma}\label{lem1.0}

Let $A$ be a subset of $\Fbb_q^2 \times \Fbb_q^2$ such that for $(m_1,m_2) \in A$ with $m_1=(u_1,v_1) , m_2=(u_2,v_2)$ we have $u_1,u_2 \ne 0$ and $(v_1,v_2) \ne (0,0)$. 
Assume the image of $A$ under the first projection to $\F_q^2$ does not contain two distinct points on the same line through the origin. 
Then the map $f \colon \SL_2(\Fbb_q) \rar \Fbb_q^3$, $\begin{pmatrix}
    a & b \\
    c & d
\end{pmatrix}\mapsto (a,d,c)$ satisfies the following properties. 
\begin{itemize}
    \item[(i)]  For any $(m_1,m_2)\in A$, the set 
    \[ \ell_{m_1 \rar m_2}:= f \l( \lo \theta \in \SL_2 (\Fbb_q) \colon \theta m_1 =m_2 \ro \r) \]
    is a line in $\Fbb_q^3$;
    \item[(ii)]  for $(m_1,m_2)\neq (m_1',m_2')$, we have $\ell_{m_1\rar m_2}\neq \ell_{m_1'\rar m_2'}$;
    \item[(iii)] $f$ is injective on the set $\SL_2(\mathbb{F}_q)\setminus\{\theta\colon \theta(e_1)=e_1\}$, here $e_1$ is the line defined by $y=0$.
\end{itemize}

\end{lemma}

\begin{proof}
Let $(m_1,m_2)\in A$. Write $m_1=(u_1,v_1)$ and $m_2=(u_2,v_2)$. The set of all $\theta=\begin{pmatrix}
    a & b \\
    c & d
\end{pmatrix} \in \SL_2 \l( \Fbb_q \r)$ such that $\theta m_1 =m_2$ is described as the solution set of the following system of three equations
\begin{equation*}
    \begin{cases}
    au_1 +bv_1 =u_2, \\
    cu_1 +dv_1 =v_2, \\
    ad-bc =1.
\end{cases}
\end{equation*}

Solving this system, we have
\[ \begin{bmatrix}
    a \\
    d \\
    c
\end{bmatrix} = \begin{bmatrix}
    \frac{u_2}{u_1} \\
    \frac{u_1}{u_2} \\
    \frac{v_2}{u_1} -\frac{v_1}{u_2}
\end{bmatrix} + b \begin{bmatrix}
    -\frac{v_1}{u_1} \\
    \frac{v_2}{u_2}  \\
    -\frac{v_1v_2}{u_1u_2}
\end{bmatrix},\]
for $b \in \Fbb_q$. Now one checks that $f$ satisfies the required properties. 
\end{proof}

\begin{lemma}\label{Lem:on_2_lines_|R_E|_leq_|E|}
Let $E\subset\F_q^2$. Suppose there are exactly $2$ lines through the origin that intersect $E-\{(0,0)\}$. Then 
$|R_E|\leq |E|$. 
\end{lemma}
\begin{proof}
Since $\SL_2(\F_q)$ acts doubly transitive on the set of lines through the origin, we can assume without loss of generality that the two lines are the two coordinate axes $e_1$ and $e_2$. Denote by $E_1$ and $E_2$ the intersections of $E-\{(0,0)\}$ with $e_1$ and $e_2$, respectively. Fix a point $m=(x,0)$ on $E_1$.  

A matrix $\theta\in\SL_2(\F_q)$ that sends $e_1\mapsto e_1$ and $e_2\mapsto e_2$ has the form $\theta=\begin{pmatrix}a&0\\0&d\end{pmatrix}$. 
Fixing the image $\theta m$ on $E_1$ uniquely determines $a$, hence also $d$. Thus, the number of $\theta$'s that preserve $E_1$ and $E_2$ is at most $|E_1|$.

Next, a matrix $\theta\in\SL_2(\F_q)$ that swaps
$e_1$ and $e_2$ has the form $\theta=\begin{pmatrix}0&b\\c&0\end{pmatrix}$. 
Fixing the image $\theta m$ on $E_2$ uniquely determines $c$, hence also $b$. Thus, the number of $\theta$'s that swap $E_1$ and $E_2$ is at most $|E_2|$.

Altogether, $|R_E|\leq |E_1|+|E_2|=|E|$. 
\end{proof}

\begin{lemma}
     Let $\mathcal{L}$ be a set of $m\geq 3$ lines through the origin in $\F_q^2$. Then the number of matrices $\theta$ in $\SL_2(\F_q)$ such that $\theta(\mathcal{L})=\mathcal{L}$ is 
     $\le 2\cdot m^3(m-1)^2$.
     \label{Lem: O(1) lines}
\end{lemma}

\begin{proof}
As in the previous lemma, we can assume that $\mathcal{L}$ contains the two coordinate lines $e_1$ and $e_2$. Consider a matrix $\theta=\begin{pmatrix} a &b\\c& d\end{pmatrix}$ in $\SL_2(\F_q)$ that acts on $\mathcal{L}$. Deciding which element in $\mathcal{L}$ goes under $\theta$ to $e_1$ amounts to fixing the ratio $c/d$ in $\F_q\cup\{\infty\}$. Similarly, deciding which element in $\mathcal{L}$ goes under $\theta$ to $e_2$ amounts to fixing the ratio $a/b$. Deciding where $\theta$ sends $e_1$ and $e_2$ determines the ratios $a/c$ and $b/d$.

Therefore, unless $a=d=0$ or $b=c=0$, the ratio $a/d$ is also determined by the choices we have made; then $\theta=\lambda\theta_0$ for a fixed matrix $\theta_0$ and a scalar $\lambda\in\F_q$. The condition $\det(\theta)=1$ now gives at most two possibilities for $\lambda$. If $a=d=0$ or $b=c=0$, we decide where $\theta$ sends an element in $\mathcal{L}-\{e_1,e_2\}$; this will determine $a/d$, and we finish as in the previous case.     
\end{proof}

\begin{proposition}
Let $E\subset \Fbb_q^2$. Let $R_E$ be the set of matrices $\theta$ in $\SL_2(\Fbb_p)$ such that $\theta (E) = E$.
Let  $m_0$ be the number of lines passing through the origin that intersect 
$E-\{(0,0)\}$ at exactly $m_1$ points (for a fixed $m_1\geq 1$). With the constant $c$ taken from Theorem \ref{lem4.2}, assume that $m_0>8+4c$. Then 
    \[|R_E|\le 16c^2(m_0m_1)^{3/2}.\]
In particular, $|R_E|\leq 16c^2|E|^{3/2}$.     
    \label{Prop:main_bound_set_of_pointss}
\end{proposition}

\begin{proof}
Replace $E$ by $E-\{(0,0)\}$ to assume without loss of generality that $(0,0)\notin E$.
Let $\ell_1,\dots,\ell_{m_0}$ be all the lines through the origin such that $|\ell_j\cap E|=m_1$. Let $E_j=\ell_j\cap E$ for $j=1,\dots,m_0$. 

Let $S$ be the set of matrices $\theta \in \SL_2(\Fbb_q)$ such that $\theta (E_j) \in \lo E_1 , \ldots ,E_{m_0}\ro , \, \forall j=1,\ldots ,m_0$. Then $R_E \subset S$. We will estimate $|S|$.

For $j=1,\dots,m_0$, fix a point $a_j\in E_j$. Let 
$B=\{a_1,\dots,a_{m_0}\}-(e_1\cup e_2)$ and $C=\left(\bigcup_{j=1}^{m_0}E_j\right)-(e_1\cup e_2)$.

Let \[\Omega=\{(u_1,u_2,\theta)\in B\times C\times S\ |\ \theta u_1= u_2\}.\]
Let $\theta\in S$. For any $u_1\in B$ such that $\theta u_1\notin e_1\cup e_1$, we have
$(u_1,\theta u_1,\theta)\in \Omega$. Therefore \begin{equation}
|\Omega|\geq (m_0-4)|S|.
\label{Eq:lower_bound_Omega}    
\end{equation}

Set $S_1=\{\theta\in S\ |\ \theta e_1\neq e_1\}$ and 
$S_2=\{\theta\in S\ |\ \theta e_1=e_1\}$. For a fixed $(u_1,u_2)\in B\times C$, there exists at most one $\theta=\begin{pmatrix}a&b\\0 &a^{-1}\end{pmatrix}\in S_2$ such that $\theta u_1=u_2$. Therefore $|\Omega\cap (B\times C\times S_2)|\leq |B||C|\leq m_0^2m_1$. 

We now estimate $|\Omega\cap (B\times C\times S_1)|$. Apply Lemma 
\ref{lem1.0} with $A=B\times C$. Let
\[L=\{\ell_{u_1\to u_2}\ |\ u_1\in B, u_2\in C\};\]
this is a set of lines in $\F_q^3$ with $|L|\leq m_0^2m_1$. 
Notice that $|\Omega\cap (B\times C\times S_1)|$ equals the number of incidences between
the lines in $L$ and the points in $f(S_1)$. 

We will see below that any plane in $\Fbb_q^3$ contains at most $2m_0$ lines from $L$. 
Once this is established, Theorem \ref{lem4.2} implies
    \[ I(L,f(S_1)) \le c\left(|L|^{3/4}|S_1|^{1/2}(2m_0)^{1/4} + |L|+|f(S_1)|\right),\]
yielding the upper bound
\begin{equation}    
|\Omega|
     \le 2c\left(m_0^{7/4}m_1^{3/4}|S|^{1/2}+ m_0^2m_1 +|S|\right).
\label{Eq:upper_bound_Omega}
\end{equation}     
Comparing the lower and upper bounds (\ref{Eq:lower_bound_Omega}) and (\ref{Eq:upper_bound_Omega}) for $|\Omega|$ under $m_0>8+4c$, we obtain 
$|S|\le 16c^2m_0^{3/2}m_1^{3/2}$, completing the proof. 

We now turn to verifying the technical condition that any plane in $\Fbb_q^3$ contains at most $2m_0$ lines from $L$. 

  First, we observe that $\ell_{u\rar v}$ and $\ell_{u\rar w}$ (elements of $L$) are not in the same plane, for any $v,w$ not in the same line passing through the origin. 
Let $u=(x_0,y_0), v=(x_1,y_1),w=(x_2,y_2)$. Then, by the proof of Lemma \ref{lem1.0}, we have 
\[ \ell_{u\rar v} = \lo  \l( \frac{x_1}{x_0},\frac{x_0}{x_1}, \frac{y_1}{x_0}-\frac{y_0}{x_1} \r) + t \l( -\frac{y_0}{x_0} , \frac{y_1}{x_1}, -\frac{y_0y_1}{x_0x_1} \r) \colon t \in \Fbb_q \ro ,\]
and
\[ \ell_{u\rar w} = \lo  \l( \frac{x_2}{x_0},\frac{x_0}{x_2}, \frac{y_2}{x_0}-\frac{y_0}{x_2} \r) + t \l( -\frac{y_0}{x_0} , \frac{y_2}{x_2}, -\frac{y_0y_2}{x_0x_2} \r) \colon t \in \Fbb_q \ro .\]
Since $v$ and $w$ lie in two distinct lines passing through the origin, we have $\frac{y_1}{x_1} \ne \frac{y_2}{x_2} $. Therefore, the line $\ell_{u\rar v}$ is not parallel to the line $\ell_{u\rar w}$. Moreover, since there exists no $\theta \in \SL_2(\Fbb_q)$ such that $\theta (u) =v, \theta (u)=w$, we have $\ell_{u\rar v}\cap \ell_{u\rar w} =\emptyset$. Therefore, $\ell_{u\rar v}$ and $\ell_{u\rar w}$ are neither parallel nor intersect, which means that $\ell_{u\rar v},\ell_{u\rar w} $ are not in the same plane. 
    
Moreover, we will show that for any $u\in B$, any 3 distinct lines $\ell_{u\rar v_1},\ell_{u\rar v_2},\ell_{u\rar v_3}$ in $L$ are not in the same plane for any $v_1,v_2,v_3$ in a line passing through the origin. We use the parametric formulas for the $\ell_{u\rar v_i}$ given in the proof of Theorem \ref{lem4.2}. Note that these $3$ lines are parallel to each other. Take a point $Z$ on one of them other than its pinpoint. A simple determinant computation shows that $Z$ and the three pinpoints do not lie in the same plane.

Now, let $W$ be a plane in $\F_q^3$. For any $a\in\{a_1,\dots,a_{m_0}\}-e_1\cup e_2$, the plane $W$ can contain at most $2$ lines of the form $\ell_{a\rar u_2}$. Therefore $W$ contains at most $2m_0$ lines from $L$.     
\end{proof}

\begin{proof}[Proof of Theorem \ref{Prop:main_bound_set_of_points}]
We denote the set of lines passing through the origin and intersecting the set $E-\{(0, 0)\}$ by $\mathcal{L}$.
If $|\mathcal{L}|=2$, the statement follows from Lemma
\ref{Lem:on_2_lines_|R_E|_leq_|E|}, so assume $|\mathcal{L}|\geq 3$. 
We partition $\mathcal{L}$ into subsets $\mathcal{L}=\bigcup_{i}\mathcal{L}_i$, where each set $\mathcal{L}_i$ is a union of $m_0^{(i)}$ lines and each contains exactly $m_1^{(i)}$ points from $E$. We now fall into the following cases:

{\bf Case 1:} If there exists $i$ such that $|\mathcal{L}_i|\ge 8+4c$, then the theorem follows directly from Proposition \ref{Prop:main_bound_set_of_pointss}. 

{\bf Case 2:} If there exists $i$ such that $3\le |\mathcal{L}_i|\le 8+4c$, then the theorem follows directly from Lemma \ref{Lem: O(1) lines}. 

{\bf Case 3:} If for each $i$, the cardinality of $\mathcal{L}_i$ is either $1$ or $2$. Since $|\mathcal{L}|\ge 3$, we can take a union of $2$ or $3$ of the $\mathcal{L}_i$'s that consists of $3$ or $4$ lines and is preserved by $R_E$. Then Lemma \ref{Lem: O(1) lines} gives the upper bound $|R_E|\le 2\times 4^5$. 

Putting these cases together, the proof is complete.
\end{proof}

\begin{proof}[Proof of Theorem \ref{main-thm}]
Apply Theorem \ref{Prop:main_bound_set_of_points}: let $C$ be such that whenever $E$ is a subset of $\F_q^2$ such that $E-\{(0,0)\}$ intersects at least $2$ lines through the origin, we have $|R_E|\leq C|E|^{3/2}$. 

Pick $c_2>C.c_1^{3/2}$.

Let $E$ be a subset of $\F_q^2$ with $|E|\leq c_1q^{\alpha}$. Suppose $|R_E|\geq c_2q^{\beta}$, where $\beta\geq 3\alpha/2$. 

Suppose that $E-\{(0,0)\}$ intersects at least $2$ lines through the origin. Then 
\[c_2q^{\beta}\leq |R_E|\leq C|E|^{3/2}\leq C(c_1q^{\alpha})^{3/2}=Cc_1^{3/2}q^{3\alpha/2},\]
which is a contradiction.
\end{proof}

\section{Proof of Proposition \ref{Prop:gp_theory}}

Suppose $|R_E|>p^{r-1}|E|$. By changing coordinates, we may assume that $e\colonequals \begin{pmatrix}
 1\\0   
\end{pmatrix}$ is in $E$. The stabilizer of $e$ under the action of $\SL_2(\F_q)$ on $\F_q^2$ is the Borel subgroup
\[B\colonequals\text{Stab}_{\text{SL}_2(\F_q)}(e)=\left\{\begin{pmatrix}1 & \alpha\\0 & 1\end{pmatrix}\ :\ \alpha\in\F_q\right\}\]
of $\SL_2(\F_q)$. In particular, $|B|=q=p^r$. 

Let $x\in E-\{(0,0)\}$. Applying the orbit-stabilizer theorem to the action of $R_E$ on $E$, we can write
\[p^{r-1}|E|<|R_E|=|R_E x|\cdot|\text{Stab}_{R_E}(x)|\leq |E|\cdot|\text{Stab}_{R_E}(x)|\]
to conclude $p^{r-1}<|\text{Stab}_{R_E}(x)|$. 
However, $\text{Stab}_{R_E}(x)$ is a subgroup of $\text{Stab}_{\SL_2(\F_q)}(x)$, which is a conjugate of $B$ and hence of size $p^r$.  Therefore $\text{Stab}_{R_E}(x)=\text{Stab}_{\SL_2(\F_q)}(x)$.

By the assumption on $E$, there exists a point $y=\begin{pmatrix}
u\\v    
\end{pmatrix}$ in $E$ with $v\neq 0$. Since $B\subset R_E$, all points
\[\begin{pmatrix}
    1 &\alpha\\ 0 &1
\end{pmatrix}\begin{pmatrix}
    u\\v
\end{pmatrix}=\begin{pmatrix}
    u+\alpha v\\v
\end{pmatrix}\]
as $\alpha\in \F_q$ belong to $E$. The lines through the origin that contain those points are all distinct. Therefore, $E$ intersects all lines through the origin in $\F_q^2$ except possibly the $y$-axis. 

For $\beta\in\F_q$, set $e_\beta=\begin{pmatrix}
    1\\ \beta
\end{pmatrix}$ and pick a point $y_\beta\in E\cap\text{span}_{\F_q}(e_\beta)$ with 
$y_\beta\neq \begin{pmatrix}
    0\\0
\end{pmatrix}$. Then
\begin{align*}
\text{Stab}_{R_E}(y_\beta)=\text{Stab}_{\SL_2(\F_q)}(y_\beta)
&=\text{Stab}_{\SL_2(\F_q)}(e_\beta)\\
&=\begin{pmatrix}
1& 0\\ \beta & 1    
\end{pmatrix}B\begin{pmatrix}
1 &0\\ -\beta & 1    
\end{pmatrix}\\
&=\left\{\begin{pmatrix}
1 & 0\\ \beta &1   
\end{pmatrix}\begin{pmatrix}
1 &\alpha\\ 0 & 1    
\end{pmatrix} 
\begin{pmatrix}
    1 &0 \\ -\beta & 1
\end{pmatrix} : \alpha\in\F_q\right\}\\
&=\left\{\begin{pmatrix}
1-\alpha\beta & \alpha \\ -\beta^2\alpha & \alpha\beta+1    
\end{pmatrix} : \alpha\in\F_q\right\}.
\end{align*}

Therefore, for each $\alpha, \beta\in \F_q$, the matrix $\begin{pmatrix}
1-\alpha\beta & \alpha \\ -\beta^2\alpha & \alpha\beta+1    
\end{pmatrix}$ belongs to $R_E$. 
Therefore, the point
\[\begin{pmatrix}
1-\alpha\beta & \alpha \\ -\beta^2\alpha & \alpha\beta+1    
\end{pmatrix}\begin{pmatrix}
    1\\0
\end{pmatrix}=\begin{pmatrix}
    1-\alpha\beta\\ -\beta^2\alpha
\end{pmatrix}\]
belongs to the $R_E$-orbit $R_E e$. 
These points are all distinct for $\alpha,\beta\neq 0$. Hence $|R_E e|\geq (q-1)^2$. 

We conclude
\[|R_E|=|R_E e|\cdot|\text{Stab}_{R_E}(e)|\geq  (q-1)^2 q>\frac{|\SL_2(\F_q)|}{2}\]
(for $q>3$). Therefore, $R_E=\SL_2(\F_q)$. But $\SL_2(\F_q)$ acts transitively on $\F_q^2-\{(0,0)\}$; this would force $E$ to contain all of $\F_q^2-\{(0,0)\}$, contradicting the assumption on $E$.     

\section{Some discussions}
One might ask if we can get stronger results when we replace Theorem \ref{lem4.2} by other point-line incidence bounds in the literature. In this section, we discuss this question. We first recall the following three incidence bounds.
\begin{theorem}[\cite{Ben, P.T.V.19}]\label{8.2}
    Let $P$ be a set of points and $L$ be a set of lines in $\mathbb{F}_q^3$. Then the number of incidences between $P$ and $L$ satisfies
    \[\left\vert I(P, L)-\frac{|P||L|}{q^2} \right\vert\ll q\sqrt{|P||L|}.\]
\end{theorem}

\begin{theorem}[\cite{K.15}]\label{8.3}
    Let $P$ be a set of $n$ distinct points and $L$ be a set of $m$ distinct lines in $\mathbb{F}_q^3$. Assume that no plane contains more than $c\sqrt{m}$ lines from $L$ for some constant $c$. Then we have 
    \[I(P, L)\ll |L||P|^{2/5}+|P|^{6/5}.\]
\end{theorem}

By a generic projection, one can show that the number of incidences between $P$ and $L$ is at most the number of incidences between a point set and a line set in two dimensions over an arbitrary fields, which is quite weak in general. Therefore, the main result in \cite{SD} implies the following. 
\begin{theorem}\label{8.4}
    Let $P$ be a set of points and $L$ be a set of lines in $\mathbb{F}_p^3$. Assume that $|P|^{7/8}<|L|<|P|^{8/7}$ and $|P|^{-2}|L|^{13}\ll p^{15}$, then we have
    \[I(P, L)\ll |P|^{11/15}|L|^{11/15}.\]
\end{theorem}
If we just plug these incidence bounds in the proof of Theorem \ref{main-thm}, without checking conditions on the sets and the underlying fields, one has 
\begin{itemize}
    \item Theorem \ref{8.2} implies $|S|\le q^2m_1$. 
    \item Theorem \ref{8.3} implies $|S|\le (m_0m_1)^{5/3}$.
    \item Theorem \ref{8.4} implies $|S|\le m_0^{7/4}m_1^{11/4}$.
\end{itemize}
All of these bounds are not better than that of Theorem \ref{main-thm}.

\section{Acknowledgements}
The authors would like to thank the Vietnam Institute for Advanced Study in Mathematics (VIASM) for the hospitality and the excellent working conditions, where part of this work was done. Thang Pham was partially supported by ERC Advanced Grant no. 882971, ``GeoScape", and by the Erd\H{o}s Center. We are very grateful to Misha Rudnev for pointing out that, over a prime ground field, a simpler group-theoretic argument can be applied, as well as for his many valuable comments.

\end{document}